\documentclass[10pt,reqno]{amsart}

\usepackage{amsfonts, amsmath, amssymb, amsthm}

\newtheorem{theorem}{Theorem}[section]
\newtheorem{corollary}[theorem]{Corollary}
\newtheorem{lemma}[theorem]{Lemma}

\newtheorem{definition}[theorem]{Definition}

\newcommand{\A}{\mathcal A}
\newcommand{\B}{\mathcal B}
\newcommand{\C}{\mathcal C}

\begin{document}
\title[Subsets of finite groups exhibiting additive regularity]
{Subsets of finite groups exhibiting additive regularity}

\author{Robert S. Coulter}
\address[Coulter]{Department of Mathematical Sciences, University of Delaware,
Newark, DE, 19716, United States of America.}
\email{coulter@math.udel.edu}

\author{Todd Gutekunst}
\address[Gutekunst]{Department of Mathematics, King's College,
Wilkes-Barre, PA, 18711, United States of America.}
\email{toddgutekunst@kings.edu}


\begin{abstract}
In this article we aim to develop from first principles a theory of sum sets
and partial sum sets, which are defined analogously to difference sets
and partial difference sets.
We obtain non-existence results and characterisations.
In particular, we show that any sum set must exhibit higher-order regularity
and that an abelian sum set is necessarily a reversible difference set.
We next develop several general construction techniques under the hypothesis
that the over-riding group contains a normal subgroup of order 2.
Finally, by exploiting properties of dihedral groups and Frobenius groups,
several infinite classes of sum sets and partial sum sets are introduced.
\end{abstract}

\maketitle

\section{Introduction}

In \cite{coultergutekunst}, the authors used versions of additive regularity of
subsets of groups to obtain new results on skew Hadamard difference sets.
For instance, by exploiting the additive regularity of skew Hadamard difference
sets we were able to completely categorise their full multiplier group, see
\cite{coultergutekunst}, Theorem 4.2.
Motivated, in part, by these results, in this article we treat sum sets and
partial sum sets as combinatorial objects in their own right.
Moreover, in keeping with the ``back to basics" philosophy we adopted in
\cite{coultergutekunst}, we approach our topic with the intention of using as
little heavy machinery as possible.
For example, we find we are able to manage without character theory, though we
readily acknowledge that, in one or two places, such theory may allow shorter,
if possibly less illuminating, proofs.

Let $G$ be a finite group, and let $S$ be a subset of $G$.
We shall be interested in counting the number of ways an element of $G$ can be
generated as the product of two elements of $S$.
If $S$ is an arbitrary subset, then we should expect some elements of $G$ to be
generated more often than others.
Indeed, some elements may be generated very often while others are not
generated at all.
If, however, the number of ways of generating elements of $G$ takes very few
values, we say the subset $S$ possesses {\em additive regularity}.
To make this concept more precise, we make the following definition.

\begin{definition}\label{pssdef}
Let $G$ be a group of order $v$ and let $S$ be a subset of $G$ with $|S| = k$.
We say $S$ is a {\bf $(v,k,\lambda,\mu)$ partial sum set} if every nonidentity
element in $S$ can be written in precisely $\lambda$ ways as a product in $S$
while every nonidentity element not in $S$ can be written in precisely $\mu$
ways as a product in $S$.
If $\lambda = \mu$, then $S$ is called a {\bf $(v,k,\mu)$ sum set}.
The numbers $(v,k,\lambda, \mu)$ are the {\bf parameters} of $S$.
\end{definition}
\noindent

Readers familiar with $(v,k,\lambda)$ difference sets will note how similar the
definition of a sum set is to that of a difference set.
Sum sets as presently defined were previously studied in \cite{lam751} and
\cite{sumnerbutson} as particular examples of ``addition sets''.
Proper acknowledgement is made wherever the current work coincides with these
papers' results.

It is easily checked that the set-theoretic complement of a sum set is a sum
set.
Specifically, if $S \subset G$ is a $(v,k,\mu)$ sum set, then $G \setminus S$
is a $(v,v-k,v-2k+\mu)$ sum set.
Hence we restrict our attention to sum sets of size $k \le \frac{v}{2}$.

Note the empty set is a $(v,0,0)$ sum set.
Also, if $g \in G$, then the singleton $\{g\}$ is a $(v,1,0)$ sum set if and
only if $o(g) \le 2$.
These sum sets and their complements are deemed trivial examples.
Henceforth all sum sets are understood to be nontrivial.

The parameters $(v,k,\mu)$ of a sum set must satisfy
\begin{equation}\label{parameq}
k^2 = \mu(v-1) + |S \cap S^{(-1)}|,
\end{equation}
where $S^{(-1)} = \{ s^{-1} : s \in S \}$.
Any triple of nonnegative integers $(v,k,\mu)$ with $v > k > \mu$ induce a
unique value for $|S \cap S^{(-1)}|$ with respect to Equation \ref{parameq}.
If that value is between $0$ and $k$, we say the triple $(v,k,\mu)$ are
{\em admissible parameters} for a sum set.
Of course, if $(v,k,\mu)$ is admissible, it is not necessarily true that there
exists a sum set with these parameters.

The primary goal of this paper is to offer a foundation for a comprehensive
theory of sum sets.
The theory is built upon the dual goals of providing theoretical construction
techniques which encompass all known examples while simultaneously deriving
nonexistence results to explain why there are no other examples.
Some of the theory is based on results from \cite{coultergutekunst}, and we
recall the relevant results in Section \ref{specsubsets}.
Nonexistence results are obtained in Sections \ref{theory} and
\ref{abeliansumsets}, where we also derive from first principles a theoretical
basis for the study of sum sets.
Some general constructions are given in Section \ref{construction1}, which we
then use to construct several infinite families of sum sets in Sections
\ref{construction2} and \ref{construction3}. 

\section{Special subsets} \label{specsubsets}

In \cite{coultergutekunst}, the authors introduce the notion of {\em special
subsets} and use them to explore the additive properties of skew Hadamard
difference sets.
These special subsets provide a useful mechanism for studying sum sets, so we
now recall some basic facts about them.

Let $S$ be a subset of $G$.
If $S = S^{(-1)}$ we say $S$ is {\em reversible}.
If $S \cap S^{(-1)} = \varnothing$ we say $S$ is {\em skew}.
A a skew subset of $G$ not properly contained in any other skew subset of $G$ is called a {\em maximal skew set}.
If $v = o(G)$ is odd, then any maximal skew set in $G$ has size $\frac{v-1}{2}$.
For any $a \in G$, the {\it special subsets} of $S$ with respect to $a$ are
\begin{align*}
\A_{a,S} &= \{ x \in S \, : \, a = xy^{-1} \text{ for some } y \in S \}, \\
\B_{a,S} &= \{ y \in S \, : \, a = xy^{-1} \text{ for some } x \in S \}, \\
\C_{a,S} &= \{ x \in S \, : \, a = xy \text{ for some } y \in S \}.
\end{align*}
The cardinality $|\A_{a,S}|$ counts the number of ways to write $a$ as a quotient in $S$, so in this context $S$ is a difference set set if and only if $|\A_{a,S}|$ is constant for all nonidentity elements $a \in G$.
Similarly $S$ is a sum set if and only if $|\C_{a,S}|$ is constant for all nonidentity elements $a \in G$.
The following lemma is a summary of Lemmas 2.1, 2.2, and 2.4 in \cite{coultergutekunst}.
\begin{lemma}\label{specsubsetslemma}
Let $S$ be a subset of the group $G$.
\begin{enumerate}
\renewcommand{\labelenumi}{(\roman{enumi})}
 \item $\A_{a,S} \cap \C_{a,S} = \varnothing$ for all $a \in G$ if and only if $S$ is skew.
 \item $\A_{a,S} = \C_{a,S}$ for all $a \in G$ if and only if $S$ is reversible.
 \item Suppose $G$ has odd order $v$, and suppose $S$ is a maximal skew set in $G$.
 Then 
\[ |\A_{a,S}| + |\C_{a,S}| = 
\begin{cases}
\frac{v-3}{2} &\text{if $a \in S$,}\\
\frac{v-1}{2} &\text{if $a \notin S$.}
\end{cases}
\] 
\end{enumerate}
\end{lemma}

It is natural to ask whether a difference set may also be a sum set or a partial sum set.
The former question is completely answered, and in the latter case we know of one family of difference sets that are also partial sum sets.

\begin{theorem}\label{rdsss}
Let $G$ be a group of order $v$, and let $S \subset G$.  Then any two of the following statements imply the third.
\begin{enumerate}
\renewcommand{\labelenumi}{(\roman{enumi})}
 \item $S$ is a $(v,k,\lambda)$ difference set.
 \item $S$ is a $(v,k,\mu)$ sum set.
 \item $S = S^{(-1)}$.
\end{enumerate}
In particular, a $(v,k,\lambda)$ difference set $S$ is also a sum set if and only if $S = S^{(-1)}$.
\end{theorem}
\begin{proof}
Suppose (i) and (ii) hold.  
Then the integers $(v,k,\lambda)$ satisfy 
\[ k^2 = \lambda(v-1) + k, \]
as $S$ is a difference set.
At the same time, the integers $(v,k,\mu)$ must satisfy
\[ k^2 = \mu(v-1) + |S \cap S^{(-1)}|, \]
as $S$ is a sum set.
If $\mu = \lambda$, then $k = |S \cap S^{(-1)}|$ and so $S = S^{(-1)}$.
If $\mu \ne \lambda$, then combining these two equations we have
\[(\mu - \lambda)(v-1) = k - |S \cap S^{(-1)}|. \]
Now $0 \le |S \cap S^{(-1)}| \le k$, so $0 \le k - |S \cap S^{(-1)}| \le k$.  
Since $v-1$ divides $k-|S \cap S^{(-1)}|$, we must have $k - |S \cap S^{(-1)}| = 0$.
Hence $|S \cap S^{(-1)}| = k$, so $S = S^{(-1)}$.

If (i) and (iii) hold, then $S$ is reversible, so by Lemma \ref{specsubsetslemma}, $|\C_{a,S}| = |\A_{a,S}|$ for all nonidentity elements $a \in G$.  
Hence $S$ is a sum set with $\mu = \lambda$.
Similarly, if (ii) and (iii) hold, then by Lemma \ref{specsubsetslemma} $S$ is a difference set with $\lambda = \mu$.

The final statement, which also follows from Corollary 2.5 in \cite{sumnerbutson}, is now clear.
\end{proof}

The next theorem appears as Theorem 3.1 in \cite{coultergutekunst}.
\begin{theorem}
A $(v,k,\lambda)$ skew Hadamard difference set is a $(v,k,\lambda, \lambda + 1)$ partial sum set.
\end{theorem}

\section{Nonexistence and restriction of admissible parameters}\label{theory}

We begin with a basic observation regarding sum sets in an arbitrary group $G$.
Whenever a sum set contains a pair of commuting elements, their product is generated at least twice.
More formally, we have the following: 

\begin{lemma}\label{mueven}
If a $(v,k,\mu)$ sum set $S \subset G$ contains a pair of commuting elements, then $\mu \ge 2$.
In particular, if $G$ is abelian, then $\mu$ is even.
\end{lemma}
\begin{proof}
The first claim is obvious.
Suppose $G$ is abelian, and suppose $S \subset G$ is a $(v,k,\mu)$ sum set with $\mu$ odd.
Form all possible products of two distinct elements of $S$.
In doing so, each element of $G$ is generated some even number of times.
To satisfy $\mu$ odd, there must then be at least one element $s \in S$ for each nonidentity $g \in G$ such that $s^2 = g$.
This is possible only if $|S| \ge v-1$, so $S$ is a trivial sum set.
\end{proof}
\noindent
This simple lemma is enough to restrict a large class of groups from admitting sum sets.

\begin{theorem}\label{oddabelian}
An abelian group of odd order admits no sum sets.
\end{theorem}
\begin{proof}
Let $G$ be an abelian group of order $v$, with $v$ odd.  Suppose $S$ is a $(v,k,\mu)$ sum set in $G$.  Since 2 does not divide $v$, the elements $x^2$ with $x \in S$ are all distinct.  It follows that the number of ways to write any such element as a product in $S$ must be odd, contradicting the fact that $\mu$ must be even.
\end{proof}

The results established thus far allow us to reject certain admissible parameters $(v,k,\mu)$, or at least to rule out certain groups of order $v$ from admitting sum sets.
To extend our ability to restrict parameters, we now consider how the existence of normal subgroups can affect whether a group may admit sum sets.
In particular, we consider how $S$ may be partitioned by the cosets of a normal subgroup.

Suppose $G$ is a group admitting a $(o(G), k, \mu)$ sum set $S$.
Further suppose that $N$ is a normal subgroup of $G$.
Let $H$ be a group isomorphic to the quotient group $G/N$ and label the cosets of $N$ in $G$ by $N_{\alpha}$, $\alpha \in H$, so that $N_{\alpha}N_{\beta} = N_{\alpha\beta}$.
For each $\alpha \in H$, set $X_{\alpha} = |S \cap N_{\alpha}|$.
Note that
\begin{equation}\label{partitionequation}
\sum_{\alpha \in H} X_{\alpha} = k.
\end{equation}

As $S$ is a sum set, any nonidentity element $g \in G$ is generated $\mu$ times as a product in $S$.
Suppose $g \in N_{\beta}$ and $g = ab$ for some elements $a,b \in S$.
If $a \in N_{\alpha}$, then we must have $b \in N_{\alpha^{-1}\beta}$.
Now if $\beta \ne 1$, then $1 \notin N_{\beta}$.
Hence, every element of $N_{\beta}$ is generated precisely $\mu$ times as a product in $S$.
But if $\beta = 1$, then each nonidentity element of $N_1$ is generated $\mu$ times as a product in $S$ while $1$ is generated $|S \cap S^{-1}|$ times.
Thus,
\begin{equation}\label{pre}
\sum_{\alpha \in H} X_{\alpha}X_{\alpha^{-1}\beta} =
\begin{cases}
\mu o(N) &\text{if $\beta \ne 1$,}\\
\mu o(N) + |S \cap S^{(-1)}| - \mu &\text{if $\beta = 1$.}
\end{cases}
\end{equation}

The quantity $|S \cap S^{(-1)}| - \mu$ in the above equation appears frequently enough to warrant its own symbol.
We may rewrite Equation \ref{parameq} as
\[ |S \cap S^{(-1)}| - \mu = k^2 - \mu v. \]
Set $n = |S \cap S^{(-1)}| - \mu$ ($= k^2 - \mu v$).
Equation \ref{pre} has the following equivalent formulation:
\begin{equation}\label{post}
\sum_{\alpha \in H} X_{\alpha}X_{\alpha^{-1}\beta} =
\begin{cases}
\mu o(N) &\text{if $\beta \ne 1$,}\\
\mu o(N) + n &\text{if $\beta = 1$.}
\end{cases}
\end{equation}
We note that as $\beta$ varies over $H$, Equation \ref{post} yields a system of $o(H)$ equations that effectively partitions Equation \ref{parameq}.
The sum of their left-hand sides is
\[ \left( \sum_{\alpha \in H} X_{\alpha} \right)^2, \]
which is simply $k^2$ by Equation \ref{partitionequation}.
Their right-hand sides, meanwhile, sum to $\mu o(N)i_G(N) + n$, which equals $\mu(o(G) - 1) + |S \cap S^{(-1)}|$.

The following is a consequence of Equation \ref{post}:

\begin{lemma}\label{muon}
Let $G$ be a group with $N \lhd G$ and $S \subset G$ a sum set.
If the center of $G/N$ contains some element which is not a square of any element in $G/N$, then $\mu o(N)$ must be even.
\end{lemma}
\begin{proof}
Set $H = G/N$ as above, and suppose $\beta \in Z(H)$ is not a square; i.e. the equation $x^2 = \beta$ has no solution in $H$.
Clearly $\beta \ne 1$.
As $\beta \in Z(H)$, 
\[N_{\alpha}N_{\alpha^{-1}\beta} = N_{\alpha^{-1}\beta}N_{\alpha} = N_{\beta} \]
for each $\alpha \in H$.
Correspondingly, whenever the term $X_{\alpha}X_{\alpha^{-1}\beta}$ appears on the left-hand side of Equation \ref{post}, so too does the (equal) term $X_{\alpha^{-1}\beta}X_{\alpha}$.
Moreover, since $\beta$ is not a square, no term of the form $X_{\gamma}^2$ appears on the left-hand side of Equation \ref{post}.
It follows that the left-hand side of Equation \ref{post} is divisible by 2, proving the claim.
\end{proof}

We now consider the possible values of $X_{\alpha}$.
That is, if a nonsimple group admits a sum set $S$, how may that sum set be distributed over the cosets of some normal subgroup $N$?
We consider the situation where the intersection sizes $X_{\alpha}$ take only two values.
That is, suppose $M \subseteq H$ such that
\[ X_{\alpha} = 
\begin{cases}
m &\text{if $\alpha \in M$,}\\
l &\text{if $\alpha \notin M$.}
\end{cases}
\]
With only two values for $X_{\alpha}$, Equation \ref{partitionequation} reduces to 
\[ m|M| + l(o(H)-|M|) = k. \]
We will count $\sum_{\alpha \in H} X_{\alpha}X_{\alpha^{-1}\beta}$ in a different way and compare the result to Equation \ref{post}.
Before counting, it will be useful to define the set $M_{\beta} = \{\beta \gamma^{-1} : \gamma \in M \}$.
Note that for any $\beta \in H$, $\alpha \in M_{\beta}$ if and only if $\alpha^{-1}\beta \in M$.
In particular, $M \cap M_{\beta} = \C_{\beta,M}$.
We have
\begin{align*} 
\sum_{\alpha \in H} X_{\alpha}X_{\alpha^{-1}\beta} 
&= \sum_{\alpha \in M} mX_{\alpha^{-1}\beta} + \sum_{\alpha \notin M} lX_{\alpha^{-1}\beta} \\
&= \sum_{\alpha \in M \cap M_{\beta}} m^2 + \sum_{\alpha \in M \setminus M_{\beta}} ml + \sum_{\alpha \in M_{\beta} \setminus M} lm + \sum_{\alpha \in H \setminus (M \cup M_{\beta})} l^2 \\
&= |\C_{\beta, M}|m^2 + 2(|M| - |\C_{\beta,M}|)ml + (o(H) - 2|M| + |\C_{\beta,M}|)l^2\\
&= |\C_{\beta, M}|(m-l)^2 + 2|M|ml + o(H)l^2 - 2|M|l^2\\
&= |\C_{\beta, M}|(m-l)^2 + 2[k+|M|l - o(H)l]l + o(H)l^2 - 2|M|l^2\\
&= |\C_{\beta, M}|(m-l)^2 + 2kl - o(H)l^2
\end{align*}
Combined with Equation \ref{post}, we have
\begin{equation}\label{eqtwoint}
|\C_{\beta,M}|(m-l)^2 + 2kl - o(H)l^2 = 
\begin{cases}
\mu o(N) &\text{if $\beta \ne 1$,}\\
\mu o(N)+n &\text{if $\beta = 1$.}
\end{cases}
\end{equation}
Equation \ref{eqtwoint} has several immediate implications, which we summarize in a theorem.

\begin{theorem}\label{twointtheorem}
$\mbox{}$
\begin{enumerate}
\renewcommand{\labelenumi}{(\roman{enumi})}
 \item $|\C_{\beta,M}| = \omega$ is constant for all $\beta \ne 1$.
 Hence, $M$ is a $(o(H), |M|, \omega)$ sum set in $H$.
 In particular, $M$ is a subgroup of $H$ if and only if $M = \{1\}$ or $M = H$.
 \item $n = (|\C_{1,M}| - \omega)(m-l)^2$.
 \item If $M = \{1\}$, then
 \[ l = \frac{k \pm \sqrt{n}}{o(H)}. \]
 In particular $n$ must be a square.
 Moreover, the parameters of the complementary sum set satisfy the same equation with the opposite sign chosen.
 \item If $M = H$, then $n=0$ and $m = \frac{\mu o(N)}{k}$.
\end{enumerate}
\end{theorem}
\begin{proof}
That $M$ is a $(o(H), |M|, \omega)$ sum set in $H$ is clear.
Hence, $M$ is a subgroup of $H$ only if $M$ is a trivial subgroup, i.e. $M = \{1\}$ or $M = H$, in which case $M$ is a trivial sum set as well.
This proves the first claim.

To prove the second claim, simply subtract Equation \ref{eqtwoint} with $\beta = 1$ from the same equation with $\beta \ne 1$.

For the third claim, note the condition $M = \{1\}$ is equivalent to $|\C_{\beta,M}| = 0$ for all $\beta \ne 1$.
Substituting into Equation \ref{eqtwoint}, we get
\[ 2kl - o(H)l^2 = \mu o(N). \]
Solving this quadratic for $l$ yields the claimed result.
That $n$ must be a square is clear since $l$ must be an integer.
Considering the complement easily verifies the remainder of the third claim.

Finally, suppose $M = H$.
Then we must have $m = l$, so the first term in Equation \ref{eqtwoint} vanishes, rendering the left-hand side independent of $\beta$.
It follows that $n=0$.
Hence, for any $\beta \in H$ we have 
\[ 2km - o(H)m^2 = \mu o(N). \]
But we also have $k = o(H)m$, and substituting yields $km = \mu o(N)$.
The last claim follows.
\end{proof}

On the surface, the most tantalizing outcome of Theorem \ref{twointtheorem} is that even distribution of a sum set over the cosets of a normal subgroup induces a sum set in the corresponding quotient group.
The sum sets in the quotient group may be trivial, however, though even then we can deduce much about the original sum set in $G$.

Consider a group $G$ which has a normal subgroup $N$ of index 2.
As there are only two cosets of $N$ in $G$, any sum set $S$ in $G$ must intersect the cosets of $N$ in at most two values.
Regardless of how $S$ is distributed among the cosets, the corresponding sum set in the quotient group will be trivial (as the quotient group has order 2).
Nonetheless, the fact that there can be at most two intersection sizes of $S$ with cosets of $N$ leads to a restriction of the possible parameters of any sum set in $G$:
\begin{theorem}\label{index2}
Suppose $S$ is a $(v,k,\mu)$ sum set in $G$.
If $G$ has a normal subgroup $N$ of index 2, then $n$ is a square.
In particular, if $k$ is odd, then $n$ must be a nonzero square.
\end{theorem}
\begin{proof}
Let $G/N = \{ N_1, N_{\beta} \}$, so $H = \{1,\beta\}$.
If $X_1 = X_{\beta}$, then necessarily $X_1 = X_{\beta} = \frac{k}{2}$ and $n=0$.
If $X_1 \ne X_{\beta}$, then without loss of generality $M = \{1\}$, hence $n$ is a square.
If $k$ is odd, then the first conclusion is impossible, and the second conclusion is possible only if $n$ is nonzero.
\end{proof}
It should be noted Sumner and Butson \cite{sumnerbutson} prove that the
parameter $n$ is necessarily a square, though possibly $n=0$.
We include the above somewhat weaker result as it follows very naturally and
directly from our discussion.

We may obtain similar results under the assumption that $G$ has a normal subgroup of index 3.
If a sum set intersects the three cosets of this subgroup in at most 2 distinct amounts, then Theorem \ref{twointtheorem} applies.
If the sum set intersects the cosets in three distinct amounts, we obtain the following.
\begin{theorem}\label{index3}
Suppose $S$ is a $(v,k,\mu)$ sum set in $G$, and suppose $G$ possesses a normal subgroup $N$ of index 3.
If $S$ intersects the cosets of $N$ in $G$ in three distinct values, then $3$ divides $k$, $|S \cap N| = \frac{k}{3}$, and $n = -3x^2$ for some integer $x \ne 0$. 
\end{theorem}
\begin{proof}
Since $N$ has index 3 in $G$, $H = \{1, h, h^2 \}$, the cyclic group of order 3.
Applying Equation \ref{post} to the cases $\beta = h$ and $\beta = h^2$, we obtain
\begin{align*}
X_{h^2}^2 + 2X_1X_h &= \mu o(N), \\
X_h^2 + 2X_1X_{h^2} &= \mu o(N), 
\end{align*}
respectively.
Combined, these equations yield
\[ X_h^2 - X_{h^2}^2 = 2X_1(X_h - X_{h^2}). \]
By assumption $X_h \ne X_{h^2}$ so we may divide by $X_h - X_{h^2}$ to obtain
\[ X_h + X_{h^2} = 2X_1. \]
But we also know $X_1 + X_h + X_{h^2} = k$, hence $|S \cap N| = X_1 = \frac{k}{3}$.

Set $X_h = \frac{k}{3} + x$ for some integer $x \ne 0$ so that $X_{h^2} = \frac{k}{3} - x$.
Substituting these values for $X_h$ and $X_{h^2}$ into $X_{h^2}^2 + 2X_1X_h = \mu o(N)$, we obtain
\[ \mu o(N) - \frac{k^2}{3} = x^2, \]
which is equivalent to the final claim.
\end{proof}

\section{Higher-order regularity and Abelian Sum Sets}\label{abeliansumsets}

Whenever we have said an element $a$ can be generated ``as a product in S'', it has been implied there exist elements $x,y \in S$ such that $a = xy$.
In other words, we have only considered the notion of writing an element as the product of two elements of our set $S$.
One might wonder whether our notion of additive regularity can be extended to include sets for which every nonidentity element of the ambient group can be expressed some constant number of times as a product of three (or four, or five, etc...) elements of the set.
Intuitively, we refer to this extended notion as ``higher-order regularity.''
 
In what follows, we demonstrate that additive regularity implies higher-order regularity.
Specifically, if $S$ is a sum set in the group $G$, then the number of ways to write elements of $G$ as the product of $j$ elements of $S$ -- for any $j \ge 2$ -- is predictable and (almost) constant.
This fact, interestingly, proves to be the key to characterizing abelian sum sets.

The calculations that follow are most easily carried out in the integral group ring $\mathbb{Z}G$.
This ring consists of all formal sums
\[ \sum_{g \in G} a_gg \]
with $a_g \in \mathbb{Z}$.
For each element $g\in G$ there is a corresponding element $1g \in \mathbb{Z}G$, though for simplicity we identify $1g$ with $g$.
Similarly, for any $a_1 \in \mathbb{Z}$ write $a_11 \in \mathbb{Z}G$ simply as $a_1$.
Each subset of $S \subseteq G$ corresponds to a group ring element $\sum_{g \in S}g$, which we write simply as $S$.
Reusing notation in this fashion is convenient because we will often be able to use facts about the group ring element $S \in \mathbb{Z}G$ to deduce facts about the subset $S \subset G$.

Addition in $\mathbb{Z}G$ is defined component-wise:
\[ \sum_{g \in G}a_gg + \sum_{g \in G}b_gg = \sum_{g \in G}(a_g+b_g)g \]
If we define $(ag)(bh) = (ab)(gh)$ for $a,b \in \mathbb{Z}$ and $g,h \in G$, then we can extend this rule via the standard distributive laws to define multiplication in $\mathbb{Z}G$.
Expressions such as $X^t$, where $X \in \mathbb{Z}G$ and $t \in \mathbb{N}$ are then understood to mean multiplication of $X$ with itself $t$ times.

Finally, for any $X = \sum_{g \in G}a_gg \in \mathbb{Z}G$ and any integer $t$, define $X^{(t)} = \sum_{g \in G}a_gg^t$.
Note that $X^{(t)} \ne X^t$ in general!
This notation has the potential to create confusion, since if $S \subseteq G$, then we may use the same notation to define the {\it set} $S^{(t)} = \{ s^t : s \in S \}$.
In this instance, the elements appearing in the {\it set} $S^{(t)}$ will be precisely those elements whose coefficient in the {\it group ring element} $S^{(t)}$ is nonzero.
However, if there exist $r$ distinct elements in $G$ whose $t^{th}$ powers equal some element $g$, then the coefficient of $g$ in the group ring element $S^{(t)}$ will be $r$, while of course $g$ will appear but once in the set $S^{(t)}$.
Note that when $t$ is relatively prime to the order of $G$, there is no potential for ambiguity.
So, for example, if $S \subseteq G$ we may write $S^{(-1)}$ to mean both the set of inverses of $S$ and the group ring element $\sum_{s \in S}s^{-1}$ without fear of confusion.

\begin{lemma}\label{powerslemma}
If $S$ is a $(v,k,\mu)$ sum set, then for any $m \ge 1$, 
\[ S^{2m} = \frac{1}{v}(k^{2m} - n^m)G + n^m. \]
\end{lemma}
\begin{proof}
We induct on $m$.
When $m = 1$, the equation reduces to $S^2 = \mu G + n$, which is correct.
Now assume the result holds for some $m \ge 1$ and consider $S^{2(m+1)}$:
\begin{align*}
S^{2(m+1)} = S^{2m}S^2 &= (\frac{1}{v}(k^{2m}-n^m)G+n^m)(\mu G + n) \\
&= \frac{\mu}{v}(k^{2m}-n^m)vG+\frac{n}{v}(k^{2m}-n^m)G+\mu n^mG + n^{m+1} \\
&= \frac{1}{v}[(k^{2m}-n^m)(\mu v+n)+(k^2-n)n^m]G + n^{m+1} \\
&= \frac{1}{v}[(k^{2m}-n^m)k^2 + (k^2-n)n^m]G + n^{m+1} \\
&= \frac{1}{v}(k^{2m+2}-n^{m+1})G + n^{m+1}. \end{align*}
This completes the induction.
\end{proof}
Since in $\mathbb{Z}G$, $GS = SG = |S|G$ for any subset $S$, we also have
\begin{corollary}\label{powerscorollary}
If $S$ is a $(v,k,\mu)$ sum set, then for any $m \ge 1$,
\[ S^{2m+1} = \frac{k}{v}(k^{2m} - n^m)G + n^mS. \]
\end{corollary}

These results are aesthetically satisfying, for they imply that additive regularity at a base level is enough to ensure regularity at all higher levels.
However, the corollary has a rather surprising consequence in the case where $G$ is abelian.
This consequence -- that every abelian sum set satisfies $S = S^{(-1)}$ -- is the main result of this section.
To prove it, we will need the following well-known fact about arithmetic in $\mathbb{Z}G$ when $G$ is abelian:
\begin{lemma}[\cite{jungnickel92} Lemma 3.3]\label{plemma}
Let $p$ be a prime.
If $G$ is abelian, then for any $S \in \mathbb{Z}G$, 
\[ S^p \equiv S^{(p)} \bmod{p}. \]
\end{lemma}
We will also use the fact that for any odd prime $p$, a natural number $n$ relatively prime to $p$ is a quadratic residue modulo $p$ if and only if $n^{(p-1)/2} \equiv 1 \bmod{p}$.

The outline of the proof is as follows: first we prove that if $S$ is an abelian sum set, then $n \ne 0$.
Next, we show $n \ne 0$ implies $n$ is a square.
From there, we are able to deduce that $S^{(p)} = S$ for infinitely many primes $p$, which leads to the result.

\begin{theorem}\label{abelianreversible}
If $S$ is a $(v,k,\mu)$ sum set in an abelian group $G$, then $S$ is reversible.
\end{theorem}
\begin{proof}
Suppose $S$ is a $(v,k,\mu)$ sum set with $n=0$.
Then $k^2 = \mu v$.
For any odd prime $p$, Corollary \ref{powerscorollary} says
\[ S^p = \left( \frac{k}{v} \right) k^{p-1}G = \mu k^{p-2} G. \]
Hence by Lemma \ref{plemma}, we have
\[ \mu k^{p-2}G \equiv S^{(p)} \bmod{p}. \]
Provided $p$ is relatively prime to both $\mu$ and $k$ this is impossible, because the group ring element $\mu k^{p-2}G$ contains some nonzero constant number of copies of each element of $G$, while $S^{(p)}$ can not do this.
Hence, $n \ne 0$.
Next, suppose $n$ is not a square.
Then there must exist some odd prime $p>v$ such that $n^{(p-1)/2} \equiv -1 \bmod{p}$ (\cite{ireland}, Chapter 5, Theorem 3).
Thus,
\begin{align*}
vS^p &= kG(k^{p-1}-n^{\frac{p-1}{2}}) + vn^{\frac{p-1}{2}}S \\
&\equiv kG(1 - (-1)) + v(-1)S \bmod{p} \\
&\equiv 2kG - vS \bmod{p}. 
\end{align*}
Now $p > v$ implies that $p>k$, so $p$ can not divide $2k$.
Thus every element of $G \setminus S$ appears in the expression $2kG - vS$, taken modulo $p$.
Also, as $n \ne 0$, we know $2k \ne v$.
It follows that every element of $G$ appears in the expression $2kG - vS$, taken modulo $p$.
But by Lemma \ref{plemma}, we know that
\[ 2kG - vS \equiv vS^{(p)} \bmod{p}. \]
As before, this is impossible, so we conclude that if $S$ is a sum set in an abelian group, then $n$ must be a square.

If $n$ is a square, then $n$ is a square modulo $p$ for any prime $p$.
Consequently, if $p > v$, 
\begin{align*} 
S^p &= \frac{k}{v}G(k^{p-1}-n^{\frac{p-1}{2}}) + n^{\frac{p-1}{2}}S \\
&\equiv \frac{k}{v}G(1 - 1) + (1)S \\
&\equiv S \bmod{p}. \end{align*}
Lemma \ref{plemma} now says that $S^{(p)} \equiv S \bmod{p}$ for any prime $p > v$.
But clearly this is possible only if $S^{(p)} = S$.

Now we are ready to show that $S = S^{(-1)}$.
Let $x \in S$.
By Dirichlet's Theorem on arithmetic progressions, there exists a prime $p>v$ such that $p \equiv -1 \bmod{o(x)}$.
We know $S^{(p)} = S$, so $x^p = x^{-1} \in S$.
The proof is complete.
\end{proof}

In the case where $G$ is nonabelian, we can not deduce so much.
One reason is that Lemma \ref{plemma}, which was our primary tool in proving Theorem \ref{abelianreversible}, does not apply.
Although nonabelian sum sets exhibit the same higher-order regularity as abelian sum sets, it is unclear whether this fact can be exploited to gain as much information about the structure of nonabelian sum sets.

We have already seen in Theorem \ref{abelianreversible} that every abelian sum set is necessarily reversible.
Thus, we now have
\begin{corollary}\label{abeliancharacterization}
An abelian sum set is a reversible difference set.
In particular, there are no cyclic sum sets.
\end{corollary}
\begin{proof}
A reversible sum set is necessarily a difference set, by Theorem \ref{rdsss}, so the first result is obvious.
The second result, first obtained by Lam \cite{lam751}, now follows immediately as there are no reversible cyclic difference sets \cite{mcfarlandma}.
\end{proof}

\section{General Constructions}\label{construction1}

In Section \ref{theory} we saw how the existence of normal subgroups of small index can affect the possible structure of sum sets.
Here we look at the opposite end of the spectrum -- normal subgroups of order 2.
While normal subgroups of small index proved useful in developing nonexistence results for sum sets, normal subgroups of small order allow for simple, generic construction techniques for sum sets.

If $N = \{1,z\}$ is a normal subgroup of $G$, then $z \in Z(G)$.
Clearly any sum set $S \subset G$ meets each coset of $N$ in either 0, 1, or 2 elements.
We highlight two possible situations which will be of importance to us.
\begin{definition}\label{typedef}
Let $N \lhd G$, $o(N) = 2$, and let $S \subset G$ be a sum set.
We say $S$ is {\bf type 1 with respect to $N$} if it does not intersect $N$ but intersects each other coset of $N$ in either 0 or 1 element.
We say $S$ is {\bf type 2 with respect to $N$} if it intersects $N$ in one element and intersects each other coset of $N$ in either 0 or 2 elements.
\end{definition}
\noindent
The motivation for these definitions will become apparent as we develop techniques for constructing sum sets.

If $S$ is a sum set, then the translate $Sg = \{ sg : s \in S \}$ is generally not a sum set.
However, if $z$ is a central involution, then we can write $a = xy$ as a product in $S$ if and only if we can write $a = (xz)(yz)$ in $Sz$.
Thus, if $z$ is a central involution, then $Sz$ is a sum set with the same parameters as $S$.
We can say slightly more in the case where $S$ is either type 1 or type 2 with respect to the normal subgroup $N = \{1,z\}$.
\begin{lemma}\label{typelemma}
Let $S$ be a sum set in $G$ and suppose $G$ possesses a normal subgroup $N = \{1,z\}$.
If $S$ is type 1 with respect to $N$, then $Sz$ is also a type 1 sum set, necessarily disjoint from $S$.
If $S$ is type 2 with respect to $N$, then $Sz$ is also a type 2 sum set, and $Sz$ differs from $S$ in precisely one element.
\end{lemma}
\begin{proof}
That $Sz$ is a sum set is clear from the preceeding discussion.
Multiplying the elements of $G$ by $z$ fixes all cosets of $N$.
Thus $S$ and $Sz$ intersect the same cosets of $N$.
If $S$ is type 1, $S$ and $Sz$ intersect those cosets in distinct elements, by definition.
If $S$ is type 2, then $Sz$ contains the same nontrivial cosets of $N$ as $S$, while if $S$ contains 1, then $Sz$ contains $z$.
If $S$ contains $z$, then $Sz$ contains 1.
\end{proof}

We now present a technique for constructing type 2 sum sets.
Essentially, the process involves ``lifting'' a partial sum set $P$ in the group $K$ to a partial sum set in the group $G = K \times \{1,z\}$, which can then be completed to a sum set by adjoining either 1 or $z$.

\begin{theorem}\label{2lift}
Let $P$ be a $(v,k,\beta - 1, \beta)$ partial sum set in $K$ not containing 1.
Then $S = P \cup Pz$ is a $(2v,2k,2\beta - 2, 2\beta)$ partial sum set in $G = K \times \{1,z\}$ if and only if $|P \cap P^{(-1)}| = \beta$.
Consequently, when $|P \cap P^{(-1)}| = \beta$, $S \cup \{1\}$ and $S \cup \{z\}$ are both $(2v, 2k+1, 2\beta)$ type 2 sum sets in $G$ with respect to $N = \{1,z\}$.
\end{theorem}
\begin{proof}
It is clear that $|S| = 2k$, so we need only prove that $S$ exhibits the claimed additive regularity.

Consider any element $a \in G \setminus N$.
Either $a \in K$ or $a = bz$ for some $b \in K$.
We prove the number of ways to write $a$ as a product in $S$ is twice the number of ways to write $a$ as a product in $P$.

First suppose $a \in K$ and let $a = xy$ be a representation for $a$ as a product in $P$.
Then $a = xy$ is also a representation for $a$ as a product in $S$ as $P \subset S$.
In addition, $a = (xz)(yz)$ is a representation for $a$ as a product in $S$.
Conversely, whenever we can write $a = xy$ as a product in $S$ it must be the case that both $x$ and $y$ are in $P$ or both are in $Pz$.
Hence, 
\[
|\C_{a,S}| = 2|\C_{a,P}| =
\begin{cases}
2\beta - 2 &\text{if $a \in S$,}\\
2\beta &\text{if $a \notin S$.}
\end{cases}
\]

Now suppose $a = bz$ for some $b \in K$.  
Whenever $b = xy$ is a representation for $b$ as a product in $P$, $a = (xz)y$ and $a = x(yz)$ are representations for $b$ as a product in $S$.
Conversely, whenever we can write $b = xy$ as a product in $S$, we must have precisely one of $x$ or $y$ in $P$ and the other in $Pz$.
Hence, 
\[
|\C_{a,S}| = 2|\C_{b,P}| =
\begin{cases}
2\beta - 2 &\text{if $a \in S$,}\\
2\beta &\text{if $a \notin S$.}
\end{cases}
\]

Thus, with the exception of $z$, all nonidentity elements of $G$ are represented as products in $S$ in the number of ways claimed.  
Note $z \notin S$ so we require $z$ to be represented as a product in $S$ in precisely $2\beta$ ways.
Suppose $z = xy$ where $x,y \in S$.
If $x \in P$, then $y = x^{-1}z \in Pz$, so necessarily $x^{-1} \in P$.
Hence the number of ways to write $z = xy$ where $x,y \in S$ and $x \in P$ is precisely $|P \cap P^{(-1)}|$.
Similarly if $x \notin P$, then we must have $y \in P$ and $x = y^{-1}z$.
Again, the number of ways to write $z = xy$ in this situation is $|P \cap P^{(-1)}|$.
Thus $z$ can be written as a product in $S$ in precisely $2|P \cap P^{(-1)}|$ ways, so $S$ is a $(2v,2k,2\beta - 2, 2\beta)$ partial sum set if and only if $|P \cap P^{(-1)}| = \beta$.

Finally if $S$ is a $(2v,2k,2\beta - 2, 2\beta)$ partial sum set, then adjoining 1 to $S$ increases the number of ways to write each element of $S$ as a product in $S$ by 2 while not affecting the number of ways to write elements not in $S$ as products in $S$.
Hence $S \cup \{1\}$ is a $(2v,2k+1,2\beta)$ sum set which is clearly type 2 with respect to $N$.
By Lemma \ref{typelemma}, $S \cup \{z\} = (S \cup \{1\})z$ is also a type 2 $(2v,2k+1,2\beta)$ sum set.
\end{proof}

This lifting process can be reversed in the sense that if $S$ is a type 2 sum set in $G$ with respect to the normal subgroup $N$, then there exists a partial sum set $P$ in the group $G/N$.
Hence there is a 2-to-1 correspondence between type 2 sum sets in groups of order $2v$ and partial sum sets with particular parameters in groups of order $v$.

\begin{theorem}\label{2project}
Let $S$ be a $(2v,2k+1,2\beta)$ sum set in $G$, type 2 with respect to the normal subgroup $N = \{1,z\}$.
Then $P = (S \setminus N)/N$ is a $(v,k,\beta-1,\beta)$ partial sum set in $G/N$.
\end{theorem}
\begin{proof}
If $S$ is type 2, then $S$ intersects $N$ in one element, and $S \setminus N$ is a $(2v,2k,2\beta - 2, 2\beta)$ partial sum set.
For any $a \in G \setminus N$, if $x \in \C_{a,S}$ then $xz \in \C_{a,S}$ as well.
Under the canonical homomorphism $g \mapsto gN$, the elements $x$ and $xz$ are both mapped to $xN$.
Hence, $|\C_{aN,P}| = \frac{1}{2}|\C_{a,S \setminus N}|$.
The result follows. 
\end{proof}

\section{Dihedral Constructions}\label{construction2}

The preceding constructions are of a general nature, and to apply them one must already be in possession of some additively regular set.
Dihedral groups provide a wealth of additively regular sets and an easily exploitable group structure for constructing them.
Throughout this section, we use the presentation
\[ D_n = \langle x, t ~|~ x^n = t^2 = 1, txt = x^{-1} \rangle \]
to denote the dihedral group of order $2n$.

We present two constructions for sum sets in $D_n$ with parameters $(2n,n-1,\frac{n-2}{2})$.
These parameters imply $n$ is even, so that $Z(D_n)$ has order 2.
The first construction yields type 1 sum sets with respect to $Z(D_n)$ while the sum sets of the second construction are type 2 with respect to $Z(D_n)$.

\begin{theorem}\label{t1construction}
Let $C_n = \langle x \rangle$ be the cyclic group of order $n$ where $n \ge 4$ is even.
Let $M$ be a maximally skew set in $C_n$ containing precisely one element from each coset of $\langle x^{\frac{n}{2}} \rangle$.
Set $S = M \cup Mt \subset D_n$.
Then $S \cup \{t\}$ and $S \cup \{x^{\frac{n}{2}}t\}$ are both $(2n, n-1, \frac{n-2}{2})$ type 1 sum sets with respect to $Z(D_n)$.
\end{theorem}
\begin{proof}
Set $z = x^{\frac{n}{2}}$.
Let $M$ be a maximally skew set in $C_n$ containing precisely one element from each coset of $\langle z \rangle$.
As $n$ is even, $z$ is the unique involution in $C_n$.
The remaining $n-2$ nonidentity elements each have inverses distinct from themselves, so $|M| = \frac{n-2}{2}$.

Fix some $y \in C_n$.
For each $m \in M$, consider the element $m^{-1}y$.
There are four possibilities:
\begin{enumerate}
 \item[1.] $m^{-1}y = 1 \mbox{~~~~~} ( \Leftrightarrow m = y)$,
 \item[2.] $m^{-1}y = z \mbox{~~~~~} ( \Leftrightarrow m = yz)$,
 \item[3.] $m^{-1}y \in M$,
 \item[4.] $m^{-1}y \in M^{(-1)}$.
\end{enumerate}
\noindent
If $y=1$ or $y = z$, the first two cases are impossible.
Hence every $m \in M$ is in either $\C_{1,M}$ or $\A_{1,M}$, and in either $\C_{z,M}$ or $\A_{z,M}$.
As $M$ is skew, $\C_{a,M} \cap \A_{a,M} = \emptyset$ for any $a \in C_n$.
We therefore have
\begin{equation}\label{t11} 
|\C_{1,M}| + |\A_{1,M}| = |\C_{z,M}| + |\A_{z,M}| = |M| = \frac{n-2}{2}.
\end{equation}

If $y \notin \{1,z\}$, then either $y$ or $y^{-1}$ is in $M$.
If $y \in M$, then $yz \notin M$ by construction, so case 2 does not occur.
Similarly, if $y \notin M$, then $yz \in M$ so case 1 does not occur.
In either situation we have
\begin{equation}\label{t12}
|\C_{y,M}| + |\A_{y,M}| = |M| - 1 = \frac{n-4}{2}.
\end{equation}

Now we move from the group $C_n$ to the group $D_n = C_n \cup C_nt$.
Form the set $S = M \cup Mt$, and form all possible products $ab$ in $S$.
There are again four distinct situations:
\begin{enumerate}
 \item[1.] $a \in M$, $b \in M$,
 \item[2.] $a \in Mt$, $b \in Mt$,
 \item[3.] $a \in Mt$, $b \in M$,
 \item[4.] $a \in M$, $b \in Mt$.
\end{enumerate}
Note in the first two cases, the product $ab$ is in $C_n$, while in the latter two cases $ab \in C_nt$.

Multiplication in $D_n$ obeys the rule $ty = y^{-1}t$, for any $y \in C_n$.
So, for example, $(m_1t)m_2 = m_1(tm_2) = (m_1m_2^{-1})t$, while $m_1(m_2t) = (m_1m_2)t$.
Hence, the number of ways to write $yt$ as a product in $S$ is the sum of the number of ways to write $y$ as a product in $M$ and the number of ways to write $y$ as a quotient in $M$.
It follows from Equations \ref{t11} and \ref{t12} that
\[ |\C_{a,S}| =
\begin{cases}
\frac{n-2}{2} &\text{if $a \in \{1,t,z,zt\}$,}\\
\frac{n-4}{2} &\text{if $a \notin \{1,t,z,zt\}$.}
\end{cases}
\]

If we adjoin the element $t$ to $S$, what new products can be generated?
The set $S$ contains none of the elements $\{1,t,z,zt\}$, so none of $\{t,z,zt\}$ gains any additional representations as products.
The identity gains one additional representation, since $t^2 = 1$.
If $a \in D_n$ is any element other than these four, then precisely one of $at$ or $a^{-1}t$ is in $S$.
If $at \in S$, then $(at)t = a$ is a new way to write $a$ as a product in $S$.
On the other hand, if $a^{-1}t \in S$ then $t(a^{-1}t) = a$ is a new way to write $a$ as a product in $S$.
Hence,
\[ |\C_{a,S \cup \{t\}}| =  
\begin{cases}
\frac{n}{2} &\text{if $a=1$,}\\
\frac{n-2}{2} &\text{if $a \ne 1$.}
\end{cases}
\]
\noindent
Thus, $S \cup \{t\}$ is a $(2n, n-1, \frac{n-2}{2})$ sum set in $D_n$.
A similar argument shows adjoining $zt$ to $S$ produces a sum set with the same parameters.
The construction guarantees that the resulting sum set misses the center of $D_n$ while containing precisely one element from each nontrivial coset of $\{1,z\} = Z(D_n)$.
Thus, these sum sets are type 1 with respect to $Z(D_n)$.
\end{proof}

While type 1 dihedral sum sets with parameters $(2n, n-1, \frac{n-2}{2})$ can be constructed in $D_n$ for all even $n \ge 4$, there is an additional restriction for type 2 dihedral sum sets with these parameters:

\begin{lemma}\label{t2lemma}
A type 2 sum set with parameters $(2n, n-1, \frac{n-2}{2})$ may exist in $D_n$ only if $n \equiv 2 \bmod{4}$.
\end{lemma}
\begin{proof}
If there exists a $(2n, n-1, \frac{n-2}{2})$ sum set in $D_n$, then $n$ is necessarily even, so we may rewrite these parameters as $(4m, 2m-1, m-1)$.
By Theorem \ref{2project}, if there exists a type 2 dihedral sum set with parameters $(4m, 2m-1, m-1)$, then there must exist a $(2m, m-1, \frac{m-3}{2}, \frac{m-1}{2})$ partial sum set in $D_{2m} / Z(D_{2m})$.
Consequently type 2 sum sets with these parameters may exist in $D_{2m}$ only if $m$ is odd, i.e. only if $n \equiv 2 \bmod{4}$.
\end{proof}

In the case where $m$ is odd, $D_{2m} \cong D_m \times Z(D_{2m})$.
Hence $D_{2m} / Z(D_{2m}) \cong D_m$, so the problem of constructing type 2 dihedral sum sets in $D_{2m}$ is equivalent to the problem of constructing partial sum sets in $D_m$ with liftable parameters.
We now show how one may construct these partial sum sets.

The procedure begins with a maximally skew set in a group $G$ of odd order, which we ``twist'' into a partial sum set in the generalized dihedral group of $G$.
That our starting set is skew guarantees (see Lemma \ref{specsubsetslemma}) the partial sum set has parameters amenable to the lifting process described by Theorem \ref{2lift}.
When the group in which we begin is cyclic of odd order $m$, then the twist yields a partial sum set in $D_m$ which is then lifted to a sum set in $D_{2m}$.

\begin{theorem}\label{t2construction}
Let $G$ be a group of odd order $m$ with $M$ a maximal skew set in $G$.
Then $S = M \cup Mt$ is a $(2m, m-1, \frac{m-3}{2}, \frac{m-1}{2})$ partial sum set in $DihG = G \rtimes \{1,t\}$.
These partial sum sets can be lifted to $(4m, 2m-1, m-1)$ sum sets in $DihG \times C_2$.
\end{theorem}
\begin{proof}
As $|M| = \frac{m-1}{2}$, $|S| = m-1$.
Suppose $g = ab$ is a representation for $g \in G$ as a product in $S$.
Then either $a,b \in M$ or $a,b \in Mt$.
By definition there are $|\C_{g,M}|$ ways to write $g = ab$ as a product in $M$.
If $a,b \in Mt$, then we have 
\[ g = (m_1t)(m_2t) = m_1(tm_2)t = m_1(m_2^{-1}t)t = m_1m_2^{-1}. \]
Hence there are $|\A_{g,M}|$ ways to write $g$ as a product in $Mt$.
Combined, there are thus $|\C_{g,M}| + |\A_{g,M}|$ ways to write $g$ as a product in $S$.

Similarly, if $y = gt$ for some $g \in G$, then the number of ways to write $y$ as a product in $S$ is $|\C_{g,M}| + |\A_{g,M}|$.
Note $y \in S$ if and only if $g \in M$.

By Lemma \ref{specsubsetslemma}, we have
\[ |\C_{y,S}| =  
\begin{cases}
\frac{m-3}{2} &\text{if $y \in S$,}\\
\frac{m-1}{2} &\text{if $y \notin S$.}
\end{cases}
\]
Thus, $S$ is a $(2m, m-1, \frac{m-3}{2}, \frac{m-1}{2})$ partial sum set in $DihG$, as claimed.
That $M$ is skew implies $1 \notin M$ and $|S \cap S^{(-1)}| = |Mt| = \frac{m-1}{2}$.
By Theorem \ref{2lift}, $S$ can be lifted into a $(4m, 2m-1, m-1)$ sum set in $DihG \times C_2$.
\end{proof}

The partial sum set $S = M \cup Mt \subset DihG$ can also be used to construct sum sets in the group 
\[ D^*_n := \langle x,t ~|~ x^n = t^4 = 1, x^t = x^{-1} \rangle. \]
Note that this group is defined by a presentation very similar to the standard presentation for $D_n$, except the element $t$ which acts by inversion on the cyclic group $\langle x \rangle$ has order 4 rather than order 2.
It is easily seen that when $n$ is odd, this group has center $Z(D^*_n) = \{1,t^2\}$, and $D^*_n / Z(D^*_n) \cong D_n$.

Suppose $S \subset D_n$ is a $(2n, n-1, \frac{n-3}{2}, \frac{n-1}{2})$ partial sum set as described in Theorem \ref{t2construction}. 
Using the natural correspondence between $D_n$ and $D^*_n / Z(D^*_n)$, define
\[ S^* = \bigcup_{x \in S} xZ(D^*_n). \]
That $S^*$ is a $(4n,2n-2,n-3,n-1)$ partial sum set follows from an argument almost identical to that used in Theorem \ref{2lift}.
Hence, adjoining either central element to $S^*$ creates a $(4n,2n-1,n-1)$ sum set in $D^*_n$.
We summarize the preceding discussion as a theorem:
\begin{theorem}\label{generalizedt2construction}
The group 
\[ D^*_n = \langle x,t | x^n=t^4=1, x^t=x^{-1} \rangle \]
admits $(4n,2n-1,n-1)$ sum sets whenever $n$ is odd.
These sum sets are all type 2 with respect to the center $\{1,t^2\}$ of $D^*_n$.
\end{theorem}

Sum sets in the groups $D_n$ and $D^*_n$ appear as Examples 6.3 and 6.4 in \cite{sumnerbutson}.

\section{Frobenius Constructions}\label{construction3}

We have seen how normal subgroups of a group $G$ affect the possible size and shape of sum sets in $G$.
In this section, we consider a family of groups that admits both sum sets and partial sum sets.
Interestingly, these sum sets and partial sum sets are constructed with respect to subgroups which are not normal.

The groups in question are all Frobenius groups, so before proceeding with the constructions, we collect some relevant information about Frobenius groups.
For a more detailed treatment see Gorenstein \cite{bgorenstein80}.
Let $H$ be a subgroup of the group $G$.
For any $g \in G$, we write $H^g$ to represent the conjugate of $H$ by $g$, i.e. $H^g = g^{-1}Hg$.
If $N_G(H) = H$ and $H^{g_1} \cap H^{g_2} = \{1\}$ whenever $H^{g_1}$ and $H^{g_2}$ are distinct conjugates of $H$ in $G$, then $G$ is called a {\it Frobenius group} and the subgroup $H$ is called a {\it Frobenius complement}.
Note any conjugate of $H$ also functions as a Frobenius complement in $G$.

If $G$ is a Frobenius group, then $G$ possesses a proper, nontrivial normal subgroup $K$, called the {\it Frobenius kernel} of G, such that $G = K \rtimes H$, where $H$ is a Frobenius complement.
It can be shown $C_G(k) \le K$ for all nonidentity $k \in K$.
Consequently, for any nonidentity $k \in K$ and $h_1, h_2 \in H$, we have $k^{h_1} = k^{h_2}$ if and only if $h_1 = h_2$.
It follows $o(H) \le o(K) - 1$ for any Frobenius group $K \rtimes H$, with equality if and only if $H$ acts regularly (i.e. sharply transitively) on the nonidentity elements of $K$.
This regular action is the key to our construction of additively regular sets in Frobenius groups.
\begin{theorem}\label{frobtheorem}
Let $G = K \rtimes H$ be a Frobenius group where $H$ acts regularly on the nonidentity elements of $K$.
Then the union of any $t$ nontrivial left cosets of $H$ is a $(o(G),t\, o(H),t^2-t,t^2)$ partial sum set.
Any $t$ nontrivial left cosets of $H$ together with $H$ is a $(o(G),(t+1)o(H),t^2+o(H),t^2+t)$ partial sum set.
\end{theorem}
\begin{proof}
Whenever an arbitrary element $kh \in G$ is expressed as a product in $G$, we have
\begin{align*}
kh &= (k_1h_1)(k_2h_2) \\ &= k_1(h_1k_2h_1^{-1})h_1h_2 \\ &= (k_1k_2^{h_1^{-1}})(h_1h_2).
\end{align*}
Hence we require (1) $k = k_1k_2^{h_1^{-1}}$, and (2) $h = h_1h_2$.

Now $K \cap H = \{1\}$, so for $k_1,k_2 \in K$ we have $k_1H = k_2H$ if and only if $k_1 = k_2$.
So the set of left cosets of $H$ in $G$ is the set $\{kH : k \in K\}$.
Choose $t$ nonidentity elements $k_1, \ldots, k_t \in K$, and set $S = k_1H \cup \cdots \cup k_tH$.
We count the number of ways to write an arbitrary element of $G$ as a product in $S$.

Let $kh \in G$.
For each $k_i$ such that $k_iH \subset S$, consider the element $k_i^{-1}k$.
Provided $k_i \ne k$, $k_i^{-1}k \ne 1$.
As $H$ acts regularly on the nonidentity elements of $K$, for any $k_j \in \{k_1,\ldots,k_t\}$ there is a unique $h_1 \in H$ such that 
\[ k_j^{h_1^{-1}} = k_i^{-1}k. \]
With $h_1$ now fixed, there is a unique $h_2 \in H$ such that $h = h_1h_2$.
We have $k = k_ik_j^{h_1^{-1}}$ and $h = h_1h_2$, so $kh = (k_ih_1)(k_jh_2)$.

If $kh \in S$, then $k$ is one of the elements $k_1,\ldots,k_t$, so there are $t-1$ choices for $k_i$ in the preceding argument, and then $t$ choices for $k_j$.
Hence $kh$ can be written as a product in $S$ in $(t-1)t = t^2-t$ ways.
If $kh \notin S$, then there are $t$ choices for $k_i$, so $kh$ can be written as a product in $S$ in $t^2$ ways.
Thus, $S$ is a $(o(G),t\, o(H),t^2-t,t^2)$ partial sum set.

Next we consider the set $S \cup H$.
We have already counted the number of ways to write any element in $G$ as a product in $S$, so now we must only consider writing an arbitrary element $kh$ as a product $kh = h_i(k_jh_j)$ and as a product $kh = (k_ih_i)h_j$.

First suppose $kh \in S \cup H$.
If $k = 1$, i.e. $kh = h \in H$, then from the previous argument there are $t^2$ ways to write $h$ as a product in $S$.
There are an additional $o(H)$ ways to write $h$ as a product in $H$, and there is no way to write $h$ as a product where one factor lives in $H$ and the other does not.
Hence, there are $t^2 + o(H)$ ways to write any $h \in H$ as a product in $S \cup H$.

If $k \ne 1$, then $kh = h_i(k_jh_j)$ if and only if $k = k_j^{h_i^{-1}}$ and $h = h_ih_j$.
As before, the regular action of $H$ on the nonidentity elements of $K$ allows us to choose $k_j$ freely from $k_1,\ldots,k_t$, whereupon $h_i$ and hence $h_j$ are uniquely determined.
This contributes $t$ additional ways to write $kh \in S$ as a product in $S \cup H$.
Finally, to write $kh$ as a product $(k_ih_i)h_j$ in $S \cup H$, then $k$ must equal $k_i$, but $h_i$ may be chosen freely and uniquely determines $h_j$.
Hence $kh \in S$ can be written as a product in $S \cup H$ in $(t^2 - t) + t + o(H) = t^2 + o(H)$ ways, the same as the number of ways to write $h \in H$ as a product in $S \cup H$.

Now suppose $kh \notin S \cup H$.
There are $t^2$ ways to write $kh$ as a product in $S$.
As previously argued, the regular action of $H$ on nonidentity elements of $K$ provides $t$ ways to write $kh = h_i(k_jh_j)$ as a product in $S \cup H$.
To write $kh$ as a product $(k_ih_i)h_j$ we require $k = k_i$.
But we are assuming $kh \notin S$, so $k \notin \{k_1,\ldots,k_t\}$.
Hence it is impossible to write $kh$ as a product $(k_ih_i)h_j$.
Thus, the total number of ways to write $kh \notin S \cup H$ as a product in $S \cup H$ is $t^2 + t$.
This completes the proof.
\end{proof}

To make use of Theorem \ref{frobtheorem}, we first must possess a Frobenius group $K \rtimes H$ where $H$ acts regularly on the nonidentity elements of $K$, or equivalently, where $o(H) = o(K)-1$.
A sufficient condition for the existence of such a group is that $o(K)$ is a prime power.

For any prime power $q$, the set of invertible affine transformations of the form $x \mapsto ax+b$ of the field $GF(q)$ forms a group $\text{Aff}(q)$ of order $q(q-1)$.
The set of maps for which $a=1$ forms a normal subgroup $K$ isomorphic to the additive group $EA(q)$ of the field, while those maps for which $b=0$ form a subgroup $H$ isomorphic to the field's multiplicative group $C_{q-1}$.
It is easily checked that $\text{Aff}(q)$ is Frobenius with kernel $K$ and complement $H$, and that $H$ acts regularly on $K$.
Hence we may apply Theorem \ref{frobtheorem} to the group $\text{Aff}(q)$.
In particular, by fixing certain values of $t$, we can use the partial sum sets of Theorem \ref{frobtheorem} to construct sum sets in certain subgroups of $\text{Aff}(q)$ and in extensions of $\text{Aff}(q)$.

First, we consider a partial sum set $P \subset \text{Aff}(q)$ consisting of a single nontrivial left coset of $H$.
By Theorem \ref{frobtheorem}, $P$ is a $(q(q-1), q-1, 0, 1)$ partial sum set.
Note that every element not in $P$ is generated precisely once as a product in $P$.
In particular, simple counting shows that 1 must be generated precisely once as a product in $P$, so $|P \cap P^{(-1)}| = 1$.
Hence $P$ is amenable to the lifting procedure described in Theorem \ref{2lift}.
That is, if $\{1,z\}$ is the cyclic group of order 2, then $P \cup Pz$ is a $(2q(q-1),2(q-1),0,2)$ partial sum set in $\text{Aff}(q) \times \{1,z\}$.
Adjoining either 1 or $z$ to $P \cup Pz$ yields a $(2q(q-1),2q-1,2)$ sum set.
We state this conclusion as a corollary to Theorems \ref{2lift} and \ref{frobtheorem}.
\begin{corollary}\label{onecoset}
For any prime power $q$, there exist sum sets with parameters $(2q(q-1),2q-1,2)$ in the group $\text{Aff}(q) \times C_2$, where $C_2$ is the cyclic group of order 2.
\end{corollary}

If $K \rtimes H$ is a Frobenius group and $J$ is a nontrivial subgroup of $H$, then $K \rtimes J$ is a Frobenius group with kernel $K$ and complement $J$ (if $J$ is the identity subgroup then $K \rtimes J \cong K$ is not Frobenius).
The Frobenius group $\text{Aff}(q)$ has complement $H \cong GF(q)^*$, so the subgroups of $H$ correspond to divisors of $q-1$.
If $d \mid (q-1)$, then the subgroup of $H$ of order $d$ is the cyclic group $C_d$ of order $d$, so we write $EA(q) \rtimes C_d$ to denote the corresponding Frobenius subgroup of $\text{Aff}(q)$. 
We now demonstrate the existence of sum sets in all Frobenius subgroups of $\text{Aff}(q)$.
\begin{theorem}\label{frobsubgroups}
For any divisor $d \ge 2$ of $q-1$, the Frobenius group $EA(q) \rtimes C_d$ admits sum sets with parameters $(qd, 2q-1, \frac{4(q-1)}{d})$.
\end{theorem}
\begin{proof}
The proof is similar to the proof of Theorem \ref{frobtheorem}, so we adopt consistent notation.
Let $K \cong EA(q)$ denote the Frobenius kernel of $\text{Aff}(q)$, and let $H$ denote the Frobenius complement.
Let $d \ge 2$ be a divisor of $q-1$, and set $G = K \rtimes C_d$.

We begin by constructing a partial sum set in $\text{Aff}(q)$. 
The induced action of $C_d < H$ on the nonidentity elements of $K$ has $(q-1)/d$ orbits each of length $d$.
Choose two elements from each orbit, say $k_1, k_2, \ldots, k_t$, where $t = 2(q-1)/d$.
Set $S = k_1H \cup \cdots k_tH$.
By Theorem \ref{frobtheorem}, $S$ is a partial sum set in $\text{Aff}(q)$ with parameters 
\[ ( q(q-1), t(q-1), t(q-1)[t(q-1)-1], t^2(q-1)^2 )  \]
We claim $S \cap G$ is a partial sum set in $G$.

Let $kh \in G$.
For each $k_i$ such that $k_iH \subset S$, consider the element $k_i^{-1}k$.
Provided $k_i$ does not equal $k$, $k_i^{-1}k \ne 1$, so $k_i^{-1}k$ lives in one of the $(q-1)/d$ orbits of $K$ induced by $C_d$.
Let $k_j$ be an element of this orbit satisfying $k_jH \subset S$.
There is a unique $h_1 \in C_d$ such that 
\[ k_j^{h_1^{-1}} = k_i^{-1}k. \]
With $h_1$ now fixed there is a unique $h_2 \in C_d$ such that $h = h_1h_2$.
We have $k = k_ik_j^{h_1^{-1}}$ and $h = h_1h_2$, so $kh = (k_ih_1)(k_jh_2)$.

If $kh \in S \cap G$, then $k$ is one of the elements $k_1, \ldots , k_t$, so there are $t - 1$ choices for $k_i$ in the preceding argument.
Since we chose precisely two elements from each orbit of the action of $C_d$ on the nonidentity elements of $K$, there are two choices for $k_j$.
Thus $kh$ can be written as a product in $S \cap G$ in precisely $2(t-1)$ ways.
If $kh \notin S \cap G$, then there are $t$ choices for $k_i$, so $kh$ can be written as a product in $S \cap G$ in $2t$ ways.
Hence $S \cap G$ is a $(qd, 2(q-1), 2(t-1), 2t)$ partial sum set in $G$.
Adjoining 1 to $S \cap G$ yields a $(qd, 2q-1, \frac{4(q-1)}{d} )$ sum set.
\end{proof}

A few remarks about Theorem \ref{frobsubgroups} are helpful.
The key step in the proof is choosing precisely two elements from each orbit of the induced action of $C_d$ on $K$.
More generally, choosing some constant number of elements, say $c$, from each orbit will induce a partial sum set in the subgroup $G$, but unless $c=2$ this partial sum set will not be completable to a sum set.
If we do not choose some constant number of elements from each orbit, then $S \cap G$ will not be a partial sum set.

When $q$ is odd, $d=2$ is always a divisor of $q-1$, but the resulting sum set in $EA(q) \rtimes C_2$ is a trivial $(2q,2q-1,2q-2)$ sum set.
When $d=q-1$ the resulting sum set is in $\text{Aff}(q)$ itself, so we have the following corollary.
\begin{corollary}\label{twocosets}
In the group $\text{Aff}(q) = EA(q) \rtimes C_{q-1}$, any two nonidentity cosets of $C_{q-1}$ together with 1 form a $(q(q-1), 2q-1, 4)$ sum set.
\end{corollary}

\bibliographystyle{amsplain}
\providecommand{\bysame}{\leavevmode\hbox to3em{\hrulefill}\thinspace}

\end{document}